\documentclass[12pt]{amsart}
\usepackage{amsmath, amssymb}
\usepackage[a4paper]{geometry}

\usepackage{tkz-graph}

\swapnumbers
\theoremstyle{definition}
\newtheorem{df}{Definition}[section]

\theoremstyle{plain}
\newtheorem{tw}[df]{Theorem}

\newtheorem{wn}[df]{Corollary}
\newtheorem{fact}[df]{Fact}

\newcommand{\ktrans}[1][k]{\ensuremath{#1}-transitive}
\newcommand{\ktranso}[1][k]{\ktrans[#1]\ orientation}

\newcommand{\TkC}[1][k]{\ensuremath{\mathrm{T}_{#1}\mathrm{C}}}

\newcommand{\kns}[2]{\ensuremath{\big({#1}\!:\!{#2}\big)}}
\newcommand{\knsu}[2]{\ensuremath{\big[{#1}\!:\!{#2}\big]}}

\newcommand{\kn}[1][k]{\kns{{#1}}{n}}
\newcommand{\knu}[1][k]{\knsu{{#1}}{n}}

\newcommand{\pwyst}[1]{\textit{#1}}

\begin{document}

\thanks{$^1$ Institute of Mathematics, Jan Kochanowski University, 15 \'Swi\c{e}tokrzyska street, 25-406 Kielce, Poland. E-mail: krzysztof.pszczola@ujk.edu.pl.} 
\author[Krzysztof Pszczo\l a]{Krzysztof Pszczo\l a$^1$ }

\title{On $k$-transitive closures of directed paths}

\begin{abstract}
In this paper we study the structure of \ktrans\ closures of directed paths and formulate several properties. 
Concept of \ktranso\ generalize the traditional concept of transitive orientation of a graph. 
\end{abstract}

\maketitle

We use the standard notation. By an \pwyst{edge} we mean an unoriented pair of vertices, and by an \pwyst{arc} we mean an oriented pair of vertices. For a given graph $G$, $V(G)$ and $E(G)$ denotes the set of its vertices and the set of its edges, respectivly. For a digraph $G$, we write $A(G)$ for the set of its arcs. By an \pwyst{oriented graph} we mean such a digraph that if $(a,b)$ is an arc, then $(b,a)$ is not. All graphs and digraphs in this paper are finite.

\section{Motivation}

Orientation of a graph $G$ is called \pwyst{transitive} if for every $(a,b)\in A(G)$, $(b,c)\in A(G)$, 
also $(a,c)\in A(G)$. 
This concept was studied by many authors in numerous papers, see the survey \cite{Kelly1985} for example. 
The concept of transitive orientation was generalized in several ways in \cite{GyarfasJacobsonKinch1988} and \cite{Tuza1994}, \cite{Hernandez-Cruz2012}, and other papers.

\bigskip

A digraph is called \pwyst{$k$-transitive} if every directed path of the length $k$ has a \textit{shortcut}
joining the beginning and the end of this path.
In other words, if $(v_0, \ldots , v_k)$ is a path in the digraph $G$, then $(v_0,v_k)\in A(G)$.

\medskip

Note that our term ``\ktrans'' coresponds to 
``$(k,1)$-transitive'' in \cite{GyarfasJacobsonKinch1988} and \cite{Tuza1994}. 

\medskip

\medskip

A \ktrans\ closure of an oriented graph $G=(V,A)$ is an oriented graph $\TkC (G)$ such that 
\begin{enumerate}
\item $V\left(\TkC (G)\right) = V(G)$, 
\item $A(G) \subset A\left(\TkC (G)\right)$, 
\item $\TkC (G)$ is 
\ktrans , \label{ktrans_indef}
\item and it has the minimal (by inclusion) set of arcs among all graphs with the above stated properties.
\end{enumerate}

\medskip 

Observe that there are oriented graphs for which the \ktrans\ closure does not exist. For example in a cyclically oriented cycle $C_{k+1}$ it is not possible to add arcs to fulfill the condition (\ref{ktrans_indef}). 

If the \ktrans\ closure does exist for some oriented graph, it is unique.

\medskip

Note that this definition is a partial answer to the point (4) in \cite[p.~41]{GyarfasJacobsonKinch1988}.

\bigskip

The aim of this paper is to describe \ktrans\ closures of directed paths.

\section{Structure of the \ktrans\ closure of the directed path}
				  
Instead of $\TkC (P_{n-1})$ we write \kn\ to denote the \ktrans\ closure of an oriented path on $n$ vertices. We label the vertices by natural numbers $1, 2, \ldots , n$ and assume that $(i,i+1)\in A(P_{n-1})$ for $1\leq i<n$. 

\medskip

Although the graph \kn\ is oriented, some of the properties will be stated for simple graphs obtained by ``forgetting'' the orientation. We belive that it is clear from the context, but to be precise, for the unoriented case we write \knu.

\medskip

In this paper by a degree sequence of a graph \knu\ we mean a sequence $(\text{deg}(1), \ldots, \text{deg}(n))$. (In/out)degree sequence of a graph \kn\ is defined in a similiar way.

\medskip

Observe that $\knu[2]$ is just the complete graph $K_n$, and $\kn[2]$ is the tournament on $n$ vertices. 

\medskip

The starting point in a construction of the \ktrans\ closure of the path $P_{n-1}$ is to add arcs $(i,i+k)$. Then we 
add arcs $(i,i+2k-1)$, at the next stage arcs $(i,i+3k-2)$, and so on. 
This construction shows that for every $k, n \in \mathbb{N}$, \kn\ is well defined. 

\medskip

The key observations are:

\begin{fact}\label{induced}
Adding one vertex to the path adds only arcs ending in this new vertex. In other words, $\kns{k}{n}$ is an induced subgraph of \kns{k}{n+1}.  \qed
\end{fact}


\begin{fact}\label{lengthofshortcut}
In the graph \kn, 
$(i,j) \in A\left(\kn\right)$ 
iff $j-i=1+l(k-1)$ for some $0\leq l \leq (n-2)/(k-1)$. 
\end{fact}

\begin{proof}
It follows directly from the construction described above that 
$$\left\{(i,j): \exists\; 0\leq l \leq (n-2)/(k-1),\, j-i=1+l(k-1)\right\} \subset A\left(\kn\right).$$ 
To show the other inclusion we use the induction on $n$. 
First observe that for $n\leq k$ all arcs are of the form $(i,i+1)$. 
Assume that all arcs in $(k:n-1)$ have the length $1+l(k-1)$ for some $l\geq 0$. To obtain a $k$-shortcut in \kn\ we need $k$ arcs, each of them of length $1+l_i(k-1)$, $l_i\geq 0$. So $j-i = 1+l_1(k-1) + 1+l_2(k-1) + \cdots + 1+l_k(k-1) = k + (k-1)(l_1 + \cdots l_k) = 1 + (k-1)(1 + l_1 + \cdots l_k)$.
\end{proof}

In Figure \ref{kns511} we present the graph \kns{5}{11}\ as an example.

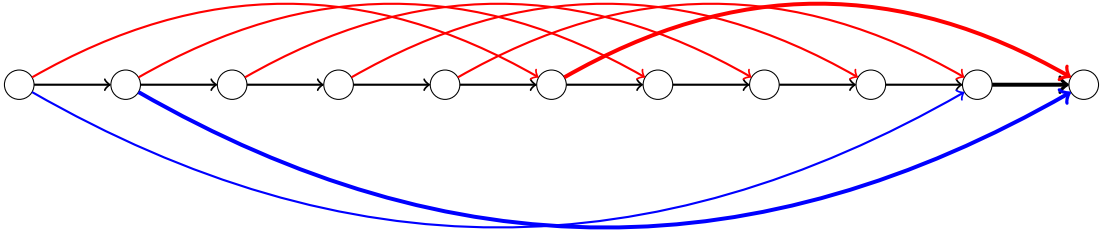
\begin{figure}
\begin{center}
\begin{tikzpicture}	
\GraphInit[vstyle=Simple]
\SetVertexNoLabel
\tikzset{VertexStyle/.style = {%
shape = circle,
fill = white,
minimum size = 5pt,draw}}
\tikzset{EdgeStyle/.style = {->}}
\Vertices[unit=1.4]{line}{v1,v2,v3,v4,v5,v6,v7,v8,v9,v10,v11}
\Edges(v1,v2,v3,v4,v5,v6,v7,v8,v9,v10)
\SetUpEdge[style = {->}, lw = 1.5pt]
\Edge(v10)(v11)
\SetUpEdge[style = {->,bend left}, 
color=red]
\Edge(v1)(v6)
\Edge(v2)(v7)
\Edge(v3)(v8)
\Edge(v4)(v9)
\Edge(v5)(v10)
\SetUpEdge[style = {->,bend left}, lw = 1.5pt,
color=red]
\Edge(v6)(v11)
\SetUpEdge[style = {->,bend right}, 
color=blue]
\Edge(v1)(v10)
\SetUpEdge[style = {->,bend right}, lw = 1.5pt,
color=blue]
\Edge(v2)(v11)
\end{tikzpicture}

\end{center}
\caption{The graph \kns{5}{11}. 
All arcs ending in the last vertex 
are drawn with thick lines.} \label{kns511}
\end{figure}

\section{Some properties}

From the observations mentioned above, we conclude several properties of graphs \kn\ and \knu.

\begin{fact}
For $n\leq k$, $\TkC (P_{n-1}) = P_{n-1}$. So for $n\leq k$, the graph \kn\ is just the path $P_{n-1}$. \qed
\end{fact}

We can observe the following block structure in indegree/outdegree sequences of graphs \kn: 

\begin{tw}\label{orienteddegreesequences}
Let $n=2+l(k-1)+m$ for some $l\in \mathbb{N}$ and $0\leq m<k-1$.
In the oriented graph \kn\ the indegree sequence is built from uniform ``blocks'' of length $k-1$ and has the form 
$$(0, \underbrace{1, \ldots , 1}_{k-1 \mathrm{\ times\ }}, \underbrace{2, \ldots , 2}_{k-1 \mathrm{\ times}}, \ldots , \underbrace{l, \ldots , l}_{k-1 \mathrm{\ times}}, \underbrace{l+1, \ldots , l+1}_{m+1 \mathrm{\ times\ }}).$$ 
Similarly, the outdegree sequence is built from uniform ``blocks'' of length $k-1$ and has the form 
$$(\underbrace{l+1, \ldots , l+1}_{m+1 \mathrm{\ times\ }}, \underbrace{l, \ldots , l}_{k-1 \mathrm{\ times}}, \ldots , \underbrace{2, \ldots , 2}_{k-1 \mathrm{\ times}}, \underbrace{1, \ldots , 1}_{k-1 \mathrm{\ times\ }},0).$$  
\end{tw}

\begin{proof}
The proof follows from Facts \ref{lengthofshortcut} and \ref{induced}. 
We prove the part concerning the indegree sequence. 
First note that the indegree of the first vertex is 0. For the next $k-1$ vertices there is no arcs ending in them other than the arcs in the initial path, so their indegree is 1. First vertex of indegree 2 is the $(k+1)$-th vertex. First vertex of indegree 3 is the $2k$-th vertex, and  first vertex of indegree $j$ is the $((j-1)(k-1)+2)$-th vertex.

The proof for the outdegree sequence is similiar; we just start from the last vertex.
\end{proof}

\begin{wn}\label{regular}
The graph $\knsu{k}{2+l(k-1)}$ is $(l+1)$-regular for every $l\in \mathbb{N}$. 
\end{wn}

\begin{proof}
This is a consequence of Theorem \ref{orienteddegreesequences}; just observe that summing up the indegree and outdegree sequences gives the constant sequence $(l+1, \ldots , l+1)$.
\end{proof}

\begin{wn}\label{twodegrees}
For every $l\in \mathbb{N}$, and for every $0< m<k-1$, all vertices of the graph $\knsu{k}{2+l(k-1)+m}$ has degree $l+1$ or $l+2$. Morover, if we put $a=l+1$ and $b=l+2$, the degree sequence is built from ``blocks'' of the form 
$$(a, \underbrace{b, \ldots b}_{m \mathrm{\ times\ }}, \underbrace{a, \ldots a}_{k-m-2 \mathrm{\ times}})$$ 
repeated to get the sequence of the length $2+l(k-1)+m$. Note that the last ``block'' has the length $m+2 \;(\mathrm{mod} (k-1))$. 
\end{wn}

\begin{proof}
This is another consequence of Theorem \ref{orienteddegreesequences}.
\end{proof}

\begin{wn}\label{howmany}
For every $l\in \mathbb{N}$, and for every $0\leq m<k-1$, in the graph $\knsu{k}{2+l(k-1)+m}$ there are $\:m(l+1)$ vertices of degree $\,l+2$ and $\;l(k-m-1)+2$ vertices of degree $\,l+1$. \qed
\end{wn}

As an example, below are the degree sequences for \ktrans[5]\ closures of the paths on $10, 11, 12, 13$ and $14$ vertices:
\begin{itemize}
 \item[] for \knsu{5}{10}: $(3,3,3,3,3,3,3,3,3,3)$; $10=2+2(5-1)+0$, $l=2$ and $m=0$, by Corollary \ref{regular} this graph is 2+1=3 regular;
 \item[] for \knsu{5}{11}: $(3,4,3,3,3,4,3,3,3,4,3)$; $11=2+2(5-1)+1$, $l=2$ and $m=1$, by Corollary \ref{twodegrees} this sequence is built from repeated blocks $(3,4,3,3)$;
 \item[] for \knsu{5}{12}: $(3,4,4,3,3,4,4,3,3,4,4,3)$; $12=2+2(5-1)+2$, $l=2$ and $m=2$, by Corollary \ref{twodegrees} this sequence is built from repeated blocks $(3,4,4,3)$;
 \item[] for \knsu{5}{13}: $(3,4,4,4,3,4,4,4,3,4,4,4,3)$; $13=2+2(5-1)+3$, $l=2$ and $m=3$, by Corollary \ref{twodegrees} this sequence is built from repeated blocks $(3,4,4,4)$;
 \item[] for \knsu{5}{14}: $(4,4,4,4,4,4,4,4,4,4,4,4,4,4)$; $14=2+3(5-1)+0$, $l=3$ and $m=0$, by Corollary \ref{regular} this graph is 3+1=4 regular. 
\end{itemize}

\medskip
Recall that by degree of a vertex $v$ in a digraph we mean a pair\\ $\left(\mathrm{indegree}(v), \mathrm{outdegree}(v)\right)$.

For oriented graphs \kn\ we can observe the following: 

\begin{wn}\label{orienteddegree}
Every constant subsequence in the degree sequence of the non regular graph \knu\ is also the constant subsequence in the degree sequence of the oriented graph \kn.  \qed
\end{wn}

For example, the degree sequence for \knsu{5}{12} is $(3,4,4,3,3,4,4,3,3,4,4,3)$, and the degree sequence for \kns{5}{12} is $((0,3),(1,3),(1,3),(1,2),(1,2),(2,2),(2,2),(2,1),$ $(2,1),(3,1),(3,1),(3,0))$.  

\medskip

Recall that an oriented graph $G$ is irregular if for every two vertices $v_i, v_j \in V(G)$, $v_i \neq v_j$, their degrees are different.

Strightforward consequence of Corollary \ref{orienteddegree} is that graphs \kn\ for $k>3$ are not irregular. The natural question is: are the graphs \kn[3] irregular? The answer is: 

\begin{tw}
Oriented graphs \kn[3] are irregular iff $n$ is odd. 
\end{tw}

\begin{proof}
By Theorem \ref{orienteddegreesequences}, pairs of vertices $2i, 2i+1$ for $1\leq i < n/2$ have the same indegree. 
Because for the even $n$ the graph \knu[3] is regular, so 
\kn[3] is not irregular if $n$ is even. 

Also by Theorem \ref{orienteddegreesequences}, $\mathrm{indegree}(2i) = \mathrm{indegree}(2i-1) +1$ for $1\leq i < n/2$. By Corollary \ref{twodegrees}, for the odd $n$ the degree sequence of the graph \knu[3] is of the form $(a,b,a,\ldots , b,a)$. So for $n$ odd, if for some two vertices their total degrees are equal then its  indegrees are different. Hence \kn[3] is irregular if $n$ is odd. 
\end{proof}

Recall that the tournament \kn[2]\ is irregular for any $n$.

\section{Density}

By density of the graph $G$, $|V(G)|=n$, we mean ratio ``number of edges in the graph $G$''/``number of edges in complete graph $K_n$''; in symbols $\mathrm{Dens}(G) = E(G)/E(K_n) = 2E(G)/(n(n-1))$. 

Recall then for every even $n$, a graph \knsu{3}{n}\ is $\frac{n}{2}$-regular. We have $\frac{1}{4} \, n^2$ edges. So for even $n$, $\mathrm{Dens}(\knsu{3}{n}) = \frac{1}{2}\frac{n}{n-1}$.

For every odd $n$, in a graph \knsu{3}{n}\ there are $\lceil\frac{n}{2}\rceil$ vertices of degree $\lfloor\frac{n}{2}\rfloor$ and $\lfloor\frac{n}{2}\rfloor$ vertices of degree $\lceil\frac{n}{2}\rceil$. We have $\frac{1}{4} \, (n^2-1)$ edges. So for odd $n$, $\mathrm{Dens}(\knsu{3}{n}) = \frac{1}{2}\frac{n+1}{n}$.

Observe that in both cases the density is bigger then $1/2$ and 
$$\lim _{n\rightarrow \infty} \mathrm{Dens}\left(\knsu{3}{n}\right) = 1/2.$$ 

\medskip

We have the following:

\begin{tw}\label{tw-density}
For $k>2$, 
$$\lim _{n\rightarrow \infty} \mathrm{Dens}\left(\knu\right) = 1/(k-1).$$  
\end{tw}

\begin{proof}
By Corollary \ref{howmany}, 
$$\mathrm{Dens}\left(\knsu{k}{(2+l(k-1)+m)}\right)=\frac{m(l+1)(l+2)+\left(2+l(k-m-1)\right)(l+1)}{n(n-1)}.$$
By standard calculation we get 
\begin{equation}\nonumber
\begin{split}
\mathrm{Dens}\left(\knsu{k}{(2+l(k-1)+m)}\right)&=\frac{(l+1)\left(m(l+2)+(2+l(k-m-1)\right)}{(2+l(k-1)+m)(1+l(k-1)+m)}=\\ 
&= 
\frac{l^2(k-1)+l(k+2m+1)+2(m+1)}{l^2(k-1)^2+l(k-1)(2m+3)+(m+1)(m+2)}.
\end{split}
\end{equation}
Recall that $n=2+l(k-1)+m$, where $k$ and $m$ are fixed, so when $l\rightarrow \infty$, then also $n\rightarrow \infty$. So $\lim _{n\rightarrow \infty} \mathrm{Dens}(\knu) = 1/(k-1).$  
\end{proof}

Obviously, for $k=2$, $\mathrm{Dens}(\knsu{2}{n}) = 1$.

\section{Open problems}

The main open problem concerning \ktrans\ closures in general, is to state what properties of an oriented graph $G$ guarantee the existence of $\TkC\left(G\right)$.

There are also some other special classes of oriented graphs, such as cycles (with different orientations) and trees, for which there is a chance to obtain interested properties for their $k$-transitive closures. 

\section*{Acknowledgments}

We acknowledge the support by the UJK grant no.\ 612439. 

Some of the results contained in this paper were presented at the 5th Polish Combinatorial Conference, B{\c{e}}dlewo, September 22--26, 2014. The author wants to express his thanks to Professor Zsolt Tuza for pointing to valuable references.


\end{document}